\documentclass[12pt]{amsart}

\usepackage[usenames,dvipsnames]{color}
\usepackage{array}
\usepackage{indentfirst,pict2e}
\usepackage{amsfonts,amssymb}
\usepackage[leqno]{amsmath}
\usepackage{amsmath}
\usepackage{amsthm}
\usepackage{latexsym}
\usepackage{youngtab, ytableau}
\usepackage{algorithm}
\usepackage{algpseudocode}
\usepackage{pifont}

	\definecolor{linkred}{rgb}{0.7,0.2,0.2}
	\definecolor{linkblue}{rgb}{0,0.2,0.6}

\usepackage[backrefs,msc-links,nobysame]{amsrefs}

\makeatletter

\def\bib@div@mark#1{%
 \@mkboth{{#1}}{{#1}}%
	}
\def\print@backrefs#1{%
 \space\SentenceSpace$\leftarrow$\csname br@#1\endcsname
}
\catcode`\'=11 
\renewcommand{\PrintAuthors}[1]{%
 \ifx\previous@primary\current@primary
  \sameauthors\@empty
 \else
  \def\current@bibfield{\bib'author}%
		  \PrintNames{}{}{\scshape #1}%
 \fi
}
\catcode`\'=12
\def\MRhref#1#2{%
 \begingroup
  \parse@MR#1 ()\@empty\@nil%
  \href{\MR@url}{\texttt{\@tempd\vphantom{()}}}%
  \ifx\@tempe\@empty
  \else
   \ \href{\MR@url}{\texttt{(\@tempe)}}%
  \fi
 \endgroup
}%
\def\MR#1{%
 \relax\ifhmode\unskip\spacefactor3000 \space\fi
 \begingroup
  \strip@MRprefix#1\@nil
  \edef\@tempa{\@nx\MRhref{MR\@tempa}{\@tempa}}%
 \@xp\endgroup
 \@tempa
}
\makeatother

\numberwithin{equation}{section}

\newtheorem{theorem}[equation]{Theorem}
\newtheorem{theorem*}{Theorem}
\newtheorem{definition}[equation]{Definition}

\newtheorem{lemma}[equation]{Lemma}
\newtheorem{corollary}[equation]{Corollary}

\newtheorem{proposition}[equation]{Proposition}

\newtheorem{example}[equation]{Example}
\newtheorem{remark}[equation]{Remark}
\newtheorem*{problem}{Problem}

\newcommand{\rank}{\operatorname{rank}}
\newcommand{\sL}{\mathfrak{sl}}

\newcommand{\rk}{\operatorname{rank}}

\begin{document}

\pagenumbering{arabic}

\title{quantum kostka and the rank one problem for $\mathfrak{sl}_{2m}$}
\author{Natalie Hobson}
\date{\today}
\maketitle

\begin{abstract}
We give necessary and sufficient conditions to specify vector bundles of conformal blocks for $\mathfrak{sl}_{2M}$ with rectangular weights which have ranks $0$, $1$, and larger than $1$. First Chern classes of rank one bundles are shown to determine a finitely generated subcone of the nef cone.    

\end{abstract}

\section{Introduction}
Given a simple Lie algebra $\mathfrak{g}$, a positive integer $\ell$, and an $n$-tuple $\vec{\lambda}$, of dominant integral weights for $\mathfrak{g}$ at level $\ell$, one can construct a globally generated vector bundle
$\mathbb{V}(\mathfrak{g}, \vec{\lambda}, \ell)$
on the moduli space  of stable $n$-pointed rational curves $\overline{\operatorname{M}}_{0,n}$, referred to as a conformal blocks vector bundle \cite{TUY}, \cite{Fakh}. The first Chern class of such vector bundles, the conformal blocks divisor $c_1(\mathbb{V}(\mathfrak{g}, \vec{\lambda}, \ell))$, nonnegatively intersects all curves on $\overline{\operatorname{M}}_{0,n}$, and so is an element of the cone of nef divisors $\operatorname{Nef}(\overline{\operatorname{M}}_{0,n})$. 

As Fakhruddin did in \cite{Fakh}, it is natural to ask whether the subcone of the nef cone generated by all conformal blocks divisors is finitely generated. 
A simpler problem, considered here for a specific set $\mathcal{S}$, consists of two questions:
\begin{problem}\label{Question}
\begin{enumerate}
\item Describe sets $\mathcal{S}:=\{ \mathbb{V}=\mathbb{V}(\mathfrak{g},\vec{\lambda},\ell) \ | \   \operatorname{rank}(\mathbb{V}(\mathfrak{g}, \vec{\lambda}, \ell))=1\}$.
\item  Show that 
$\mathcal{C}(\mathcal{S}):=\text{ConvHull}\{c_1(\mathbb{V})  \ | \ \mathbb{V} \in \mathcal{S}\}$
 is finitely generated.
 \end{enumerate}
\end{problem}

For example, nontrivial level one bundles in type A have rank one \cite[5.2.1]{Fakh}; the cone generated by $\mathcal{G} := \{\mathbb{V}=\mathbb{V}(\sL_{r+1},\vec{\lambda},1) \ | \ \mathbb{V} \ne 0\}$ is finitely generated \cite[1.1]{GiansiracusaGibney}.  

We examine this question for $\mathfrak{g}=\mathfrak{sl}_{2m}$, giving necessary and sufficient conditions which determine when the rank of $\mathbb{V}(\mathfrak{sl}_{2m}, \vec{\lambda}, \ell)$ is zero, one, and larger than one in case $\vec{\lambda}= (c_1\omega_m, ..., c_n\omega_m)$, which we refer to as \textit{rectangular weights}.

\begin{theorem} \label{mainIntro} 
Let $\mathbb{V}_m=\mathbb{V}(\mathfrak{sl}_{2m},(c_1\omega_m, ..., c_n\omega_m), \ell)$ as above with $c_1 \geq c_2 \geq ... \geq c_n$ and $\sum_{i=1}^n c_i = 2(k\ell+p)$, for some integers $p$ and $k$ such that $1\leq p \leq \ell$ and $k\geq 0$.  Denote $\Lambda = \sum_{i=2k+2}^n c_i$ where $\Lambda:=0$ if $2k+2 > n$.  Then
\begin{enumerate}
\item $\rk(\mathbb{V}_m)=0$ iff $\Lambda <p$; 
\item $\rk(\mathbb{V}_m)=1$ iff either
\begin{enumerate}
\item $\Lambda =p$, or
\item $\Lambda >p$ and the weight content is maximal (see Def \ref{MaxContent}); and
\end{enumerate}
\item $\rk(\mathbb{V}_m)>1$ iff $\Lambda >p$,  and the weight content is not maximal.
\end{enumerate}
\end{theorem}

To prove Theorem \ref{mainIntro}, we first establish the statement for $\mathfrak{sl}_2$ bundles in Section \ref{rank} and then apply a rank scaling statement for $\sL_{2m}$ rank one bundles with rectangular weights, Proposition \ref{bgkscaling}, which relies partially on work of Belkale, Gibney, and Kazanova (see Proposition \ref{GITScaling}). Further details of their proposition, and the converse in our specific situation, is outlined in Section 7.

In particular, Theorem \ref{mainIntro} characterizes the ranks of all $\mathfrak{sl}_2$  bundles.  
The set $\mathcal{S}$ of rank one bundles described by Theorem \ref{mainIntro} contains an infinite number of elements, including all $\sL_2$ level one bundles, and so following work of Fakhruddin,  $\mathcal{C}(\mathcal{S})$  forms a full dimensional subcone of the nef cone $\operatorname{Nef}(\overline{\operatorname{M}}_{0,n})$.  We prove this cone $\mathcal{C}(\mathcal{S})$  is finitely generated.

\begin{theorem}Let $\mathcal{S}=\{\mathbb{V}=\mathbb{V}(\sL_{2m},(c_1\omega_m, ..., c_n\omega_m), \ell): \rk(\mathbb{V})=1\}$.  
Then
$\mathcal{C}(\mathcal{S}):=\text{ConvHull}\{c_1(\mathbb{V})  \ | \ \mathbb{V} \in \mathcal{S}\}$
 is finitely generated.
\end{theorem}
To prove this, we show the following result\footnote{See  Theorem \ref{divisors} for a precise statement.}, which gives  $\mathcal{C}(\mathcal{S}) \subset \mathcal{C}(\mathcal{G})$.
\begin{theorem}\label{ParaDivisors}
Each element of $\mathcal{C}(\mathcal{S})$ can be expressed explicitly as an effective linear combination of level one divisors determined by the weights.
\end{theorem}

Theorem \ref{mainIntro}, proved in Section $4$,  also gives information about vector bundles of conformal blocks for the Lie algebra $\mathfrak{sp}_{2\ell}$ at level one. Using
$$ \operatorname{rank}(\mathbb{V}(\mathfrak{sl}_2, \vec{\lambda}, \ell))=\operatorname{rank}(\mathbb{V}(\mathfrak{sp}_{2\ell}, \vec{\lambda}^T, 1)),$$ from \cite[5.2.3]{Fakh} where
 $\vec{\lambda} = (c_1\omega_{1}, ..., c_n\omega_{1})$ and $\vec{\lambda}^T = (\omega_{c_1}, ..., \omega_{c_n})$,  the same conditions on $(c_1,\ldots,c_n)$ given in Theorem \ref{mainIntro} also determine when $\mathbb{V}=\mathbb{V}(\mathfrak{sp}_{2\ell},(\omega_{c_1}, ..., \omega_{c_n}), 1)$ has rank zero, one, and greater than one.

\subsection{Outline of the paper and a note on our methods}

\medskip
In Section 2 we state a method for computing ranks of conformal blocks vector bundles known as \textit{Witten's dictionary} and relevant background on \textit{Kostka numbers} necessary to compute ranks. In Section 3 we combine these tools to our specific situation in order to state an explicit method for computing ranks by means of counting Young tableaux. We then give several definitions relevant to our approach and describe an algorithm for creating Young tableaux. In Section 4 we prove our main result, Theorem \ref{mainIntro}, by applying the algorithm from Section 3. In Section 5, we describe the decomposition of rank one vector bundles described in Theorem \ref{mainIntro}. In the final section, we give a proof of the scaling statement, Proposition \ref{GITScaling}, communicated to us by the authors of \cite{bgk}.

\medskip

\section{Background}\label{background}
The normalized dominant integral weights of level $\ell$ for the simple Lie algebra $\mathfrak{sl}_{r+1}$ are parametrized by Young diagrams, or partitions, in $r \times \ell$. Such a weight is thus given by an $r$-tuple of integers $\lambda =  (\lambda^{(1)}, ..., \lambda^{(r)})$ such that $\ell \geq \lambda^{(1)} \geq ... \geq \lambda^{(r)} \geq 0$. The Young diagram associated to this weight contains $r$ rows with $\lambda^{(i)}$ boxes in the $i^{th}$ row. We call $|\lambda| = \sum_{i=1}^r \lambda^{(i)}$ the \textit{area} of the weight $\lambda$. With this notation, a basis of the fundamental dominant weights $\omega_i$ for $\mathfrak{sl}_{r+1}$ is written $\omega_j =(1,\ldots,1,0,\ldots,0)$ where $|\omega_j|=j$. If $\vec{\lambda} = (\lambda_1, ..., \lambda_n)$ is a collection of $n$ dominant integral weights, then the total area of the weights is the sum of all weights in our collection, $|\vec{\lambda}| = \sum_{i=1}^n |\lambda_i|$. In order to define a conformal blocks vector bundle with $\mathfrak{sl}_{r+1}$, weight $\vec{\lambda}$, and level $\ell$ it is necessary that $|\vec{\lambda}| = (r+1)(\ell +s)$ for some integer $s$. We refer to a collection of $n$ weights for the Lie algebra $\mathfrak{sl}_{2m}$ which are all multiples of the same fundamental dominant weight $\omega_m$ as \textit{rectangular}. The weight vector $\vec{\lambda} = (c_1\omega_m, ..., c_n\omega_m)$ we often write as $(c_1, ..., c_n)$ when the fundamental dominant weight is clear from context.

\subsection{Witten's Dictionary and Computing Ranks}
Schubert classes in the cohomology ring $H^*(Gr(r+1,\mathbb{C}^{m}); \mathbb{Z})$ and the (small) quantum cohomology ring $QH^*(Gr(r+1,\mathbb{C}^{p}); \mathbb{Z})$ are indexed by partitions $\lambda \subset (r+1) \times m-(r+1)$. From \cite{BelkaleWittenDic} we have the following connection between ranks of $\mathbb{V}(\mathfrak{sl}_{r+1}, \vec{\lambda}, \ell)$ and cohomology computations. The following result is often referred to as \textit{Witten's Dictionary}.

\begin{proposition}\label{witten} \textbf{Witten's Dictionary.}\\
To compute $\rk(\mathbb{V}(\mathfrak{sl}_r, \vec{\lambda}, \ell))$ with $|\vec{\lambda}| = (r+1)(\ell +s)$, we consider the following two cases:

\begin{itemize}
\item[(1)]  If $s \leq 0$, then $\rank(\mathbb{V}(\mathfrak{sl}_{r+1}, \vec{\lambda}, \ell))$ is equal to the coefficient of the class of a point $\sigma_{(\ell + s)\omega_{r+1}} = \sigma_{(\ell +s, ..., \ell+s)} $ in the classical product: $$\sigma_{\lambda_1} \cdot ... \cdot \sigma_{\lambda_n} \in H^*(Gr(r+1, \mathbb{C}^{r+1+\ell+s})).$$
\item[(2)]   If $s > 0$, then $\rk(\mathbb{V}(\mathfrak{sl}_{r+1}, \vec{\lambda}, \ell))$ is equal to the coefficient of $q^s\sigma_{\ell \omega_{r+1}}$ in the quantum product: $$\sigma_{\lambda_1} * ... * \sigma_{\lambda_n} * \sigma^s_{\ell \omega_1}\in QH^*(Gr(r+1, \mathbb{C}^{r+1+\ell})),$$
where $\sigma_{\ell \omega_{r+1}} = \sigma_{\ell, ..., \ell}$ and $\sigma^s_{\ell \omega_1} = \sigma^s_{(\ell, 0, ..., 0)} $.
\end{itemize}

\end{proposition}

\begin{remark}\label{invariantorder} 
The above products are commutative and the rank of $\mathbb{V}(\mathfrak{sl}_r, \vec{\lambda}, \ell)$ is invariant under the ordering of the weights in $\vec{\lambda}$. 
\end{remark}

The following scaling statement to prove Theorem \ref{mainIntro}. See Section \ref{scalingproof} for a proof of a more general statement in one direction communicated to us by Belkale, Gibney, and Kazanova.

 \begin{proposition}\label{bgkscaling} Let $\mathbb{V}_m=\mathbb{V}(\mathfrak{sl}_{2m},(c_1\omega_m, ..., c_n\omega_m), \ell)$, then 
 
 $\operatorname{rank}\mathbb{V}_m = 1$ if and only if $\operatorname{rank}\mathbb{V}_1 =1$. Furthermore,
 $c_1(\mathbb{V}_m)=m \ c_1(\mathbb{V})$.
\end{proposition}

\subsection{Kostka numbers}\label{kostka}
According to \cite{Bertram}, to compute the classical or quantum product of a collection of $n$ simple classes (i.e., classes of the form $\sigma_{(c_i, 0, ..., 0)}$ where $c_i$ is a positive integer) with the class we use the following rule, where $\mu = (c_1, ..., c_n).$

\begin{enumerate}
\item \label{eq:ckostka} $\text{Classical: } \sigma_{(c_1, 0, ..., 0)} \cdot ... \cdot \sigma_{(c_n, 0, ..., 0)} \cdot \sigma_{\lambda} = \sum K^{\nu}_{\lambda, \mu}\sigma_{\nu} \in H^*(Gr(r+1, \mathbb{C}^{r+1+\ell+s}))$, 
where we sum over partitions $\nu \subset (r+1) \times (\ell+s)$ such that $|\nu| = |\lambda| + \sum_i^n c_i.$ 

The coefficient $K^{\nu}_{\lambda, \mu}$ is called the \textit{classical Kostka number}. This number is equal to the number of Young tableaux on the shape $\nu/\lambda$ with content $\mu$. 

\item \label{eq:qkostka} $\text{Quantum: } \sigma_{(c_1, 0, ..., 0)} * ... * \sigma_{(c_n, 0, ..., 0)} * \sigma_{\lambda} = \sum K^{\nu}_{\lambda, \mu, m}(r+1, \ell)q^m\sigma_{\nu} \in QH^*(Gr(r+1,\mathbb{C}^{r+1+\ell}); \mathbb{Z})$,
where we sum over partitions $\nu \subset (r+1) \times \ell$ and $m \geq 0$ such that $|\nu| + m(r+1+\ell) = |\lambda| + \sum_i^n c_i$ and $\mu= (c_1, ..., c_n)$.

The coefficient $K^{\nu}_{\lambda, \mu, m}(r+1, \ell)$ is called the \textit{quantum Kostka number}. This number is equal to the number of proper Young tableaux with shape $\nu[m]/\lambda$ and content $\mu$. \\
\end{enumerate}

The Young diagram $\nu[m]$ is obtained by adding $m$ rim hooks to $\nu$ each of size $r+1+\ell$ beginning in column $\ell$.  We obtain the shape $\nu[m]/\lambda$ by removing $\lambda$ from the top left of $\nu[m]$.  We say that such a diagram, $\nu[m]/\lambda$, has \textit{content} \footnote{\textit{Content} here is used as defined in \cite{Bertram}. } $ \mu$, if the boxes of $\nu[m]/\lambda$ are labeled with $c_1$ 1's, $c_2$ 2's, ..., and $c_n$ n's such that the rows are weakly increasing in value (left to right) and the columns are strictly increasing in value (top to bottom). To help eliminate confusion between the $c_i$ and $i$ in the above, we define terms for these objects. We also give a definition important in distinguishing vector bundles which have rank one for maximal reasons, we further elaborate on this in Section \ref{rank}.

\begin{definition} For a collection of content $\mu= (c_1, ..., c_n)$, we call $c_i$ the \textit{amount} or size of content and $i$ the \textit{flavor} of the content. 
\end{definition}

\begin{definition}\label{MaxContent}
A collection of content containing $n$ flavors $\mu= (\mu_1, ..., \mu_n)$, is  \emph{$\ell$-maximal} (or \emph{maximal} when $\ell$ is clear) if $n-3$ or more flavors have amounts of size $\ell$.
\end{definition} 

\begin{remark}\label{MaxConentRemark}
With $k$ and $p$ integers such that $\sum_{i=1}^n = 2(k\ell + p)$, the content $(c_1, ..., c_n)$ being maximal implies one of two situations depending the parity of $n$. If $n$ is odd then $2k+2=n$ and $c_i =\ell$ for $i = 1, ..., 2k$ and if $n$ is even then $2k+2=n$ and $c_i=\ell$ for $i=1, ..., 2k-1$. We will use this in both Theorem \ref{mainIntro} and Theorem \ref{divisors}. See proof of Lemma \ref{case2leftarrow}. 
\end{remark}

A Young diagram with content values is referred to as a $\textit{semistandard}$ Young tableau. Such a semistandard tableau is $\textit{proper}$ if for all positive integers $q$ (such that $(r+1)+q$ is a row in the first column of the Young diagram) the content in row $(r+1)+q$ and column 1 is either greater than or equal to the content in row $q$ and column $\ell$ or else such a box is not in the Young diagram.

\section{Definitions and Lemmas}\label{definitionandlemma}

\subsection{Witten's Dictionary and classical Kostka applied to $\mathfrak{sl}_2$}\label{ckostka} 

To compute the rank of any $\mathfrak{sl}_2$ bundle with $s\leq0$, Witten's Dictionary \ref{witten} and the classical equation in \ref{kostka} gives

\begin{equation}\label{eq:crank}
\rk(\mathbb{V}(\mathfrak{sl}_2, \ell, (c_1\omega_1, ..., c_n \omega_1))) =  K^{(\ell+s) \omega_2}_{(c_1, ..., c_{n})},
\end{equation}

\noindent where $K^{(\ell+s) \omega_2}_{(c_1, ..., c_{n})}$ is the number of Young tableau with shape $(\ell+s) \omega_2$ and content $(c_1, ..., c_n)$. 

\subsection{Witten's Dictionary and quantum Kostka applied to $\mathfrak{sl}_2$}\label{qkostka} 
To compute the rank of any $\mathfrak{sl}_2$ bundle with $s>0$, Witten's Dictionary \ref{witten} and quantum equation in\ref{kostka} gives

\begin{equation}\label{eq:qrank}
\rk(\mathbb{V}(\mathfrak{sl}_2, \ell, (c_1\omega_1, ..., c_n \omega_1))) =  K^{\ell \omega_2}_{\ell \omega_1, (c_1, ..., c_n, \ell, ..., \ell), 1},
\end{equation}

\noindent where $K^{\ell \omega_2}_{\ell \omega_1, (c_1, ..., c_n, \ell^{s-1}), 1}$ is the number of proper Young tableau with shape $\nu[s]/\lambda$ with $\nu = \ell \omega_2$, $\lambda = \ell \omega_1$ (see Figure \ref{fig:nu}), and content $(c_1, ..., c_n, \ell^{s-1})$ (superscript to denote the number of content of size $\ell$).

\begin{figure}[h]
\setlength{\unitlength}{0.14in} 
\centering 
$$\nu = \ell \omega_2 = 
\ytableausetup
{mathmode, boxsize=1em}
\begin{ytableau}
\text{} & \text{} & \text{} & \none &  \none[\dots]
& \none & \text{} & \text{} \\
\text{} & \text{} & \text{} & \none & \none[\dots]
& \none & \text{} & \text{} 
\end{ytableau} $$ 

$$\lambda = \ell \omega_1=
\ytableausetup
{mathmode, boxsize=1em}
\begin{ytableau}
\text{} & \text{}  & \text{}  &  \none  & \none[\dots]
& \none  & \text{} & \text{} 
\end{ytableau}, $$ 

\caption{Young diagrams in (\ref{eq:qrank})} 
\label{fig:nu} 
\end{figure}

To carry out such a computation, we must first consider the shape $\nu[s]/\lambda,$ which requires more analysis than in the classical case (\ref{eq:crank}). Recall $s$ refers to the number of $(\ell+2)$-rim hooks added to $\nu$. We give an example to motivate the general shape of $\nu[s]/\lambda$.\\

\textbf{Example:}\label{examplebigshape}
 Let $\ell = 4$, we construct the Young tableau $\nu[s]/\lambda$ for $s=5$ and $\nu$ and $\lambda$ as given above. Each of the five rim hooks added to $\nu$ is indicated with a shade. The shape $\lambda$ has been deleted from $\nu[5]$.
$$
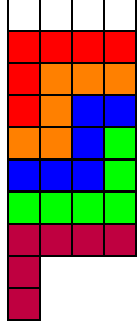
\ytableausetup
{mathmode, boxsize=1em}
  \begin{ytableau}
\text{} & \text{} & \text{} & \text{}  \\
*(red) & *(red) &*(red) &*(red)  \\
*(red) & *(orange) &*(orange) &*(orange) \\
*(red) & *(orange)  &*(blue) &*(blue) \\
*(orange) & *(orange) &*(blue) &*(green) \\
*(blue) &*(blue) & *(blue) & *(green) \\
*(green) & *(green) &*(green) & *(green)\\
 *(purple) & *(purple) & *(purple) & *(purple) \\
  *(purple) \\
   *(purple) 
\end{ytableau} 
$$

We generalize this example to describe the shape of $\nu[s]/\lambda$ for any integer $s>0$ and $\ell >0$. 

\begin{lemma}\label{bigdiagramshape}
Let $\nu = \ell \omega_2$, $\lambda = \ell \omega_1$ and $s>0$. Let $p, k$ be integers such that $1 \leq p \leq \ell$ and $s = (k-1)\ell + p$ (the relevance of such integers will become apparent in \ref{numberyoungtab} when analyzing the area of weight vectors).  Then $$\nu[s]/\lambda=(\ell^{s+2k-1}, p, p).$$
\end{lemma}

\begin{proof} 
Following the construction of the Young tableau in Example \ref{examplebigshape}, we see that if we add $0 < s < \ell$ rim hooks, we add $s$ new rows with $\ell$ boxes in each row (full rows) and two rows, each with $s$ boxes in each row.  When we add $s = \ell $ rim hooks, we add $s=\ell$ full rows and two rows of size $\ell$, that is two additional full rows.  When $p=0$, this will close up the shape giving $(k-1)(\ell +2)+1$ full rows.\\
\end{proof}

\begin{figure}[h]
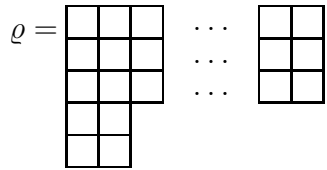

\setlength{\unitlength}{0.14in} 
\centering 
$$\varrho =
\ytableausetup
{mathmode, boxsize=1em}
  \begin{ytableau}
\text{} & \text{} & \text{} &  \none &  \none[\dots]
&  \none & \text{} & \text{} \\
\text{} & \text{} & \text{} &  \none &  \none[\dots]
&  \none &  \text{} & \text{} \\
\text{} & \text{} & \text{} & \none &  \none[\dots]
&  \none &  \text{} & \text{} \\
\text{} & \text{} \\
\text{} & \text{}
\end{ytableau} 
$$
\caption{Young diagram $\varrho=(\ell^{2k}, p^2)$ for rank computations } 
\label{fig:varrho} 
\end{figure}

\subsection{Young tableaux that count ranks of $\mathfrak{sl}_2$ bundles}
\begin{lemma}\label{numberyoungtab}
For $\mathbb{V}_1$ as in Theorem \ref{mainIntro}, the rank of $\mathbb{V}_1$ is equal to the number of proper Young tableau with content $\mu = (c_1, ..., c_n)$ on the shape $(\ell^{2k}, p^2)$ (see Figure \ref{fig:varrho}). We denote this shape $\varrho$. It is a vertical concatenation of two rectangular shapes, one of dimension $\ell \times (2k) $ and one of dimension $p \times 2$. 
\end{lemma}

\begin{proof}
For $s \leq 0$ this is a restatement of the description in \ref{eq:crank}.

For $s>0$, we use the description in \ref{eq:qrank} and Remark \ref{invariantorder}. We must analyze the number of proper Young tableau on the diagram $\nu[s]/\lambda$ from Lemma \ref{bigdiagramshape} and content $(\ell^{s-1}, c_1, ..., c_n)$. By reordering the content values, we can make the content of smallest flavor (from 1 to s-1) each have size $\ell$. Since rows must be in strictly decreasing in order, we must fill in the first $s-1$ rows $\nu[s]/\lambda$ with the first $s-1$ content flavors of size $\ell$. The remaining boxes (which must be filled with content $(c_1, ..., c_n)$) now has shape $\varrho$ described in the lemma statement.

Thus, the rank computation from Equation \ref{eq:crank} and \ref{eq:qrank} has been reduced to counting the number of proper Young tableau on $\varrho$ as claimed in the lemma. 
\end{proof}

\begin{remark}\label{proper}
With any Young diagram as $\varrho$ described in this lemma, the proper condition is equivalent to the flavor in the first column and $p$ row to be great than or equal to the content in the final column and $p-2$ row. This means that for a given tableau, if any flavor $i$ is contained only in \textit{at most} two consecutive rows, the tableau will be proper. \\
\end{remark}
\subsection{Methods for filling Young diagrams with content}
We now describe two methods of placing content in a Young diagram to produce. We utilize both methods to construction Young tableaux with shape $\varrho$ in the proof of Theorem \ref{mainIntro} to show when $\rk(\mathbb{V}_m) >0$.

Let $\lambda$ denote a partition which fits inside an $r \times \ell$ box. We denote $B_{(a,b)}$ the box of $\lambda$ in row $a$ and column $b$. We consider boxes in a diagram to have \textit{lexicographical} ordering with row and column. That is, $B_{(a,b)} \leq B_{(a', b')}$ if and only if $a < a'$ or else ($a=a'$ and $b \leq b'$). We refer to the boxes of $\lambda$ as being \text{higher} or \text{lower} if they are visually displayed in that manner and \text{larger} or \text{smaller} if we are referring to this ordering. 

\begin{definition}\label{lowrow}
Define the \emph{low-row} of a diagram $\lambda$ to be all of the boxes of $\lambda$, $B_{(a,b)}$, such that $B_{(a+1, b)}$ is \textit{not} a box of $\lambda$. We use $l_{\lambda} = (l_{\lambda}^{(1)}, ..., l_{\lambda}^{(k)})$ to denote the sizes of the low-row of $\lambda$. That is, $l_{\lambda}^{(j)}$ is the size of the $j^{th}$ low-row of $\lambda$.
\end{definition}

 With this definition we require $l_{\lambda}^{(1)} \neq 0$; this convention considers the highest (vertical) row of a low-row as the first row.  We omit the subscript $\lambda$ if the diagram to which we are referring to is clear. Note that $\sum_{i=1}^r l_{\lambda}^{(i)} = \max_{i\in \{1, ..., r\}} \{ \lambda^{(i)}\}$ and $k \leq r$. That is, the sum of all boxes in the low-row of $\lambda$ will be equal to the longest row of $\lambda$ and the number of rows contained in the low-row cannot be larger than the total number of rows of $\lambda$.
\begin{example}
For the Young diagram $\lambda = (4, 4, 3, 3)$, the low-row is the collection of boxes $(1, 0, 3)$ shaded in the figure.  

$$
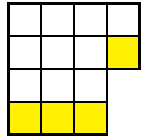
\ytableausetup
{mathmode, boxsize=1em}
  \begin{ytableau}
\text{} & \text{} & \text{} & \text{}\\
\text{} & \text{} & \text{} & *(yellow) \text{}\\
\text{} & \text{} & \text{}\\
*(yellow)\text{} & *(yellow)\text{} & *(yellow) \text{}\\
\end{ytableau} 
$$
\end{example} 

With a partial filling of a Young diagram $\lambda$, we define $\lambda(i)$ to be the Young diagram obtained from $\lambda$ by removing all boxes which contain content of flavors $i$ or larger \footnote{We again refer the reader to \cite{Bertram} for more background definitions related to Young diagram and tableau}. For $\lambda$ in $r \times \ell$ and content $\mu = (\mu_1, ..., \mu_n)$ such that $\sum_{i=1}^n \mu_i = |\lambda|$ we describe two algorithms for placing flavors $\mu$ in the boxes of $\lambda$.  

\begin{algorithm}
\caption{Forward fill method}
\label{forward}
From $i =1$ to $n$ (in increasing order) place all $\mu_i$ content of flavor $i$ in $\lambda$ in the smallest empty boxes of $\lambda$ following the lexicographical ordering of the boxes of $\lambda$.
\end{algorithm}

\begin{example}\label{forward}
Let $\lambda= (7, 7, 7, 7, 5,5)$ and $\mu =(7, 6, 6, 6, 6, 6,1)$ the following is the result of the forward fill method for placing $\mu$ in $\lambda$.

$$
\ytableausetup
{mathmode, boxsize=1em}
  \begin{ytableau}
\text{1} & \text{1} & \text{1} & \text{1} & \text{1} & \text{1} & \text{1}\\
\text{2} & \text{2} & \text{2} & \text{2} & \text{2} & \text{2}  & \text{3} \\
\text{3} & \text{3} & \text{3} & \text{3} & \text{3}  & \text{4}& \text{4} \\
\text{4} & \text{4} & \text{4} & \text{4} & \text{5}  & \text{5}& \text{5} \\
\text{5} & \text{5} & \text{5} & \text{6} & \text{6}\\
\text{6} & \text{6} & \text{6} & \text{6} & \text{7}
\end{ytableau} 
$$

\end{example}

\begin{algorithm}
\caption{Reverse fill method}
\label{reverse}
From $i =n$ to $1$ (in decreasing order) place all content $\mu_i$ of flavor $i$ in the largest boxes of the low-row of $\lambda(i+1)$ (the ordering on boxes in a low row is inherited by the ordering of the boxes of $\lambda$).
\end{algorithm}

\begin{example}\label{reverse}
Let $\lambda= (7, 7, 7, 7, 5,5)$ and $\mu =(7, 6, 6, 6, 6, 6,1)$ the following is the result of the Reverse Fill method for placing $\mu$ in $\lambda$.

$$
\ytableausetup
{mathmode, boxsize=1em}
  \begin{ytableau}
\text{1} & \text{1} & \text{1} & \text{1} & \text{1} & \text{1} & \text{1}\\
\text{2} & \text{2} & \text{2} & \text{2} & \text{2} & \text{2}  & \text{3} \\
\text{3} & \text{3} & \text{3} & \text{3} & \text{3}  & \text{4}& \text{4} \\
\text{4} & \text{4} & \text{4} & \text{4} & \text{5}  & \text{5}& \text{6} \\
\text{5} & \text{5} & \text{5} & \text{5} & \text{6}\\
\text{6} & \text{6} & \text{6} & \text{6} & \text{7}
\end{ytableau} 
$$
\end{example}

\subsection{Notation and definitions for describing the decomposition of $\mathfrak{sl}_{2m}$ divisors}

The following weight vector and vector bundle is used in Theorem \ref{divisors}.

\begin{definition}\label{vweightvector}
For $\mathbb{V}_m=\mathbb{V}(\mathfrak{sl}_{2m},(c_1\omega_m, ..., c_n\omega_m), \ell)$ and $k, p$ given as in Theorem \ref{mainIntro}, with $A, B \subset \{0, ..., n\}$ we define the weight vector and vector bundle, $$\vec{v}_{A,B} := (v_1\omega_m, ..., v_n \omega_m) \text{ and }$$  $$V_{a,b}:=\mathbb{V}(\mathfrak{sl}_{2m}, \vec{v}_{a,b}, 1)$$ where $v_i= 1$ if $i \in \{1, ..., 2k+1\} -A \cup B$ and $v_i = 0$ otherwise.
\end{definition}

As an example of this definition, for $n=9$, $k=2$, and $m = 2$, then the weight vector with $9$ weights $(\omega_2, \omega_2, 0, \omega_2, \omega_2, 0 , 0, \omega_2, 0)$ can be written as $\vec{v}_{3, 8}$.

\section{Rank classification of $\mathfrak{sl}_2$ conformal blocks}\label{rank}

We prove Theorem \ref{mainIntro}. First, for the case $m=1$ we give counting arguments showing the sufficient conditions of the size of the ranks in each of the cases for $\Lambda := \sum_{2k+2}^n c_i$. In this case, and when $\Lambda \geq p$, we produce a Young tableau on the Young diagram $\varrho$ (see Figure \ref{fig:varrho}). We show that this tableau can be modified in a simple way to produce a new tableau whenever $\Lambda > p$ and our content is not maximal. Applying Lemma \ref{numberyoungtab} we conclude the theorem for $m=1$.  By scaling, Proposition \ref{bgkscaling}, we conclude Theorem \ref{mainIntro}.   

\subsection{Counting arguments to analyze ranks}\label{counting}

Let $\mathbb{V}_1$ and $\Lambda$ as in Theorem \ref{mainIntro} and $\varrho$ and $\mu$ be a Young diagram and content as in Lemma \ref{numberyoungtab}. We give sufficient conditions for the ranks of the families described in Theorem \ref{mainIntro}.

\begin{lemma}\label{sumlessp} 
If $\Lambda < p$, then $\rk(\mathbb{V}_1) = 0$.
\end{lemma}

\begin{proof}
Applying Lemma \ref{numberyoungtab}, we show there are no possible Young tableau with shape $\varrho$ and content $\mu$.

In order to produce a Young tableau, the largest content flavors must be placed in rows with larger boxes than smaller flavors. Since $c_1 \geq ... \geq c_n$ and $\Lambda < p$, we must fill the final $2k+2$ row (which contains $p$ boxes) with the largest flavors $(c_{2k+2}, ..., c_n)$. However, since the sum of such amounts is less than $p$, we will not entirely fill this last row with this content. After placing the content $(c_{2k+2}, ..., c_n)$ the remaining collection of empty boxes creates a shape with a column of length $2k+2$. Such a column cannot be filled with the remaining content, $(c_1, ..., c_{2k+2})$, in strictly decreasing order. Thus, no such Young tableau exists.
\end{proof}


\begin{lemma}\label{sumequalp}
If $\Lambda = p$, then $\rk(\mathbb{V}_1) = 1$.
\end{lemma}

\begin{proof}
Applying Lemma \ref{numberyoungtab}, we show there is only one possible construction of a Young tableau with shape $\varrho$ and content $\mu$.

The content flavors in the sum $\Lambda= p$ are the largest in flavor and contain all but $2k+1$ content flavors. As we discussed in Lemma \ref{sumlessp}, no Young tableau can be produced on $\varrho(2k+2)$ if this shape contains a column of length $2k+2$. This forces the content $(c_{2k+1}, ..., c_n)$ to be placed in the $2k+2$ row of $\varrho$ and since $\Lambda=p$, $\varrho(2k+2)$ will have no column of length $2k+2$. After placing content $(c_{2k+1}, ..., c_n)$ in this way, the remaining empty boxes creates the shape $\varrho(2k+2) = (\ell^{2k}, p)$. 

To finish constructing a Young tableau each flavor in the remaining content $(c_1, ..., c_{2k+1})$ must be present in the first $p$ columns of length $2k+1$ of $\varrho(2k+1)$. We show this is indeed possible by showing $c_i \geq p$ for $i=1$ to $2k+1$.

Since we have $\sum_{i=1}^{n} c_i = \sum_{i=1}^{2k+1} c_i + \sum_{i=2k+2}^{n} c_i = \sum_{i=1}^{2k+1} c_i  +p$ and $\sum_{i=1}^{n} c_i =2k\ell +2p$, we have that $\sum_{i=1}^{2k+1} c_i  = 2k\ell +p$. If we assume for some $i$ in this sum that $c_i < p$, we obtain the following bound on this sum (and using that each $c_i \leq \ell$), $2k\ell +p = \sum_{i=1}^{2k+1} c_i  < (2k)\ell + p,$ which contradicts the first equality. 

To continue constructing a Young tableau, we fill the first $p$ columns of $\varrho(2k+1)$ with content $(c_1, ..., c_{2k+1})$. There is only one way to place such $2k+1$ flavors in strictly decreasing order into columns of length $2k+1$. After this placement, the empty boxes remaining form the shape $\tilde{\varrho}(2k+2):= ((\ell-p)^{2k})$ and the remaining content becomes $(c_1-p, ..., c_{2k+1}-p)$.


%
%

To continue constructing a proper Young tableau, we show that we must place the remaining content in $\tilde{\varrho}(2k+1)$ using the forward fill method \ref{forward}. We explain that this method for filling the remaining boxes is the only way of doing so in order to create a Young tableau.

For all $i$, $c_i \leq \ell$, which gives $c_i -p\leq \ell-p$ and so the amount of each remaining content flavor is less than the length of the rows $\tilde{\varrho}(2k+1)$ . We will never be in the case of having two consecutive rows with the same content flavor and so the columns will be strictly increasing and the rows will be weakly increasing. A tableau with content has been created.

We argue that in this last step, this was the only way to fill $\tilde{\varrho}(2k+1)$ with content $(c_1-p, ..., c_{2k+1}-p)$. The content amounts are weakly decreasing and each amount is smaller than the length of the rows of $\tilde{\varrho}(2k+1)$. This guarantees that any Young tableau with shape $\tilde{\varrho}(2k+1)$ and content $(c_1-p, ..., c_{2k+1}-p)$, has the additional property that the only possible content flavors in the $i^{th}$ row are $i$ and $i+1$. The only way to modify the content placed in the described method (and maintain the weakly increasing property of the rows) would be to switch two flavors of content in consecutive rows. To do so would require switching content of flavor $i$ in row $i+1$ with content of flavor $i-1$ in row $i$. This would not maintain the necessary property just described. 

In each step of the above construction, there was only one way to place content. The Young tableau produced is the only such which can be constructed on $\varrho$ with content $\mu$. 

This tableau is proper due to Remark \ref{proper} and the property that content in consecutive rows of $\tilde{\varrho}(2k+1)$ is as described above.
\end{proof}

\begin{example}\label{sumpexample}
For the vector bundle $\mathbb{V}(\mathfrak{sl}_2, (6,6,5,5,5,2,1), 6),$ we have $k=2$, $p=3$, and $\Lambda= \sum_{i=6}^7 c_1 =3=p$. The following is the only Young tableau which can be produced on the shape $\varrho = (6, 6, 6, 6, 3, 3)$ with content $\mu=(6,6,5,5,5,2,1)$. This shows $\rk(\mathbb{V})=1$. The content in the sum $\Lambda$ is shaded.

$$
\ytableausetup
{mathmode, boxsize=1em}
  \begin{ytableau}
\text{1} & \text{1} & \text{1} & \text{1} 
& \text{1} & \text{1} \\
\text{2} & \text{2} & \text{2} &\text{2} 
& \text{2} & \text{2} \\
\text{3} & \text{3} & \text{3} & \text{3} 
& \text{3} & \text{4}\\
\text{4} & \text{4} & \text{4} & \text{4} 
& \text{5} & \text{5}\\
\text{5} & \text{5} & \text{5}\\
*(yellow) \text{6} & *(yellow)\text{6} &*(yellow) \text{7}
\end{ytableau} .
$$

\end{example}

\begin{lemma}\label{case2leftarrow}
If $\Lambda  > p$ and weight content is maximal (Definition \ref{MaxContent}), then $\rk(\mathbb{V}_1) =1$. 
\end{lemma}

\begin{proof}
First, observe that the condition on the content being maximal forces the $n-2k-1$ flavors in the sum $\sum_{i=2k+2}^{n} c_i$ to contain either one or two terms.

This is due to the following observations. First, since we are assuming $\sum_{i=2k+2}^{n} c_i > p$, this implies that $n-2k-1 \geq 1$ (i.e., the sum is not empty, or zero) and particularly $c_{n-2}, c_{n-1} \geq c_n \geq p$ (since the sum is strictly larger than $p$). 
The content $(c_{n-2}, c_{n-1}, c_n)$ must have the following bound on its sum, $$2p \leq c_{n-2}+c_{n-1}+c_n \leq 2p +2 \ell.$$ Furthermore, since $\sum_{i=1}^{n} c_i =  \ell(n-3) + \sum_{i=n-2}^{n} c_i = 2k\ell + 2p$  we have two cases for $\ell(n-3)$.  These conditions give that $n-3=2k$ ($n$ is odd) or $n-3 = 2k-1$ ($n$ is even). 

We show there is one and only one way to create a Young tableau in each case.

\textbf{Case 1:}  $n$ is even and $n-3=2k$.  We fill in the highest $n-3$ rows of the Young diagram $\varrho$ from Lemma \ref{numberyoungtab} with the first $n-3$ content flavors of size $\ell$. In this case, after placing this content, the empty boxes create the shape $\tilde{\varrho}= (p, p)$. 

%

We can now translate the question of the number of tableaux with $(c_{n-2}, c_{n-1}, c_n)$ and shape $\tilde{\varrho}$ to ranks using Witten's Dictionary \ref{witten}. Using this, we obtain that the number of such tableaux is equal to $\rk\mathbb{V}(\mathfrak{sl}_2, (c_{n-2}, c_{n-1},\allowbreak c_n), p)$. Using fusion rules for $\mathfrak{sl}_2$ with three weights \cite{ags}, the rank of such a bundle is one, giving the existence of exactly one Young tableau on $\tilde{\varrho}$ with three content flavors in the amounts $(c_{n-2}, c_{n-1}, c_n)$. 

As each such flavor in $(c_{n-2}, c_{n-1}, c_n)$ is larger than the first $n-3$ flavors of size $\ell$ in the original content, using the unique tableau on $\tilde{\varrho}$ with $(c_{n-2}, c_{n-1},\allowbreak c_n)$ and the described placement of the first $n-3$ content flavors, a Young tableau is produced. Each step of the construction is unique. The final tableau produced is the only possible. 

By Remark \ref{proper} this tableau is proper.  

\begin{example} \label{evensumlargerp}
For the vector bundle $\mathbb{V}(\mathfrak{sl}_2, (6,6,6,6,2,2,2), 6),$ we have $k=2$, $p=3$, and $\Lambda=\sum_{i=6}^7 c_i = c_6 + c_7 = 2+2=4>p$. The following is the only Young tableau which can be produced on $\varrho$ with content $(6,6,6,6,2,2,2)$. This shows $\rk(\mathbb{V})=1$. The content from $\Lambda$ are shaded. This can be compared to Example \ref{sumpexample}.
$$
\ytableausetup
{mathmode, boxsize=1em}
  \begin{ytableau}
\text{1} & \text{1} & \text{1} & \text{1} 
& \text{1} & \text{1} \\
\text{2} & \text{2} & \text{2} &\text{2} 
& \text{2} & \text{2} \\
\text{3} & \text{3} & \text{3} & \text{3} 
& \text{3} & \text{3}\\
\text{4} & \text{4} & \text{4} & \text{4} 
& \text{4} & \text{4}\\
\text{5} & \text{5} & *(yellow) \text{6}\\
*(yellow) \text{6} & *(yellow)\text{7} &*(yellow) \text{7}
\end{ytableau} .
$$



%

\end{example}

\textbf{Case 2:}  $n$ is even and $n-3 = 2k-1$. Again, we fill in the highest $n-3$ rows of the Young diagram $\varrho$ from Lemma \ref{numberyoungtab} with the first $n-3$ content flavors of size $\ell$. In this case, empty boxes creates the shape $\tilde{\varrho} = (\ell, p, p)$ and our remaining content is $(c_{n-2}, c_{n-1}, c_n)$. 

To continue constructing a tableau, the columns of $\tilde{\varrho}$ with three rows must be filled with distinct flavors. In this case, $2k+2 = n$ and by assumption, $\Lambda = c_n > p$ which gives $c_{n-2} \geq c_{n-1} \geq c_n \geq p$, showing placement of the three distinct flavors is possible in columns of $\tilde{\varrho}$ with three rows. After this placement, the remaining shape is a single row. The remaining content is forced to be placed in weakly increasing order.

\begin{example} \label{oddsumlargerp}
For the vector bundle $\mathbb{V}(\mathfrak{sl}_2, (5,5,5,5,5,3,3,3), 5),$ we have $k=3$, $p=2$, and $\Lambda = \sum_{i=8}^8 c_i = c_8 = 3 > p$. The following is the only Young tableau which can be produced with content $(5,5,5,5,5,3,3,3)$ on shape $\varrho=(\ell^{6}, 2^2)$. This shows $\rk(\mathbb{V})=1$. The content from the sum $\Lambda$ is highlighted in yellow. 

$$
\ytableausetup
{mathmode, boxsize=1em}
  \begin{ytableau}
\text{1} & \text{1} & \text{1} & \text{1} 
& \text{1}  \\
\text{2} & \text{2} & \text{2} &\text{2} 
& \text{2} \\
\text{3} & \text{3} & \text{3} & \text{3} 
& \text{3} \\
\text{4} & \text{4} & \text{4} & \text{4} 
& \text{4} \\
\text{5} & \text{5} & \text{5} & \text{5} 
& \text{5} \\
\text{6} & \text{6} & \text{6} & \text{7} 
& *(yellow) \text{8} \\
\text{7} & \text{7}\\
 *(yellow)\text{8} &*(yellow) \text{8}
\end{ytableau} 
$$
\end{example}

\end{proof}

Note that these two lemmas can also be shown by using a method known as ``plussing.'' See Remark 2.3 in \cite{bgk} for more details on this method.

With these three Lemmas above, we have shown all necessary conditions in the statement of Theorem \ref{mainIntro} for $m=1$. 


\subsection{Construction of Young tableaux to show necessary conditions.}  

We now demonstrate how to construct and modify tableaux to make conclusions about when ranks are necessarily greater than one. We use definitions and notation for $\mathbb{V}_1$, $\Lambda$, $\varrho$, and $\mu$ as in Theorem \ref{mainIntro} and Lemma \ref{numberyoungtab}.

\begin{proposition} \label{combinedfillproperyoungtab}
If $\Lambda \geq p$ then a combination of the forward \ref{forward} and reverse \ref{reverse} fill methods of placing content in $\varrho$ will produce a proper Young tableau.
\end{proposition}

We first analyze the low-row of the diagrams $\varrho(i)$ which are produced after carrying out the reverse fill method with $(c_i, ..., c_n)$ on $\varrho$.\\

\begin{lemma} \label{sizeoflowrow}
Denote by $(l^{(1)}_i, l^{(2)}_i, l^{(3)}_i)$ the sizes of the low-row of $\varrho(i)$ after placement of $(c_i, ..., c_n)$ using reverse fill. We will always have $0 \leq l^{(1)}_i \leq \ell-p$, $0 \leq l^{(3)}_i \leq p$ and one of the two situations will occur for the low-row of $\varrho(i-1)$ after placement of $c_{i-1}$, $0 \leq l^{(3)}_{i-1}  \leq l^{(3)}_i \leq p$ or $0 < l^{(1)}_{i-1} \leq l^{(1)}_i  \leq \ell-p$.
\end{lemma}

This lemma states that after any stage in the reverse fill method on $\varrho$, after placing content $(c_i, ..., c_n)$, the diagram $\varrho(i)$ will have a low-row of dimension $(l^{(1)}_i, l^{(2)}_i, l^{(3)}_i)$ with $l^{(1)}_i  \leq \ell-p$ and $ l^{(3)}_i \leq p$, indicating the shape of $\varrho(i)$ has a final row containing $ l^{(3)}_i$ boxes, the second to last row containing $l^{(2)}_i + l^{(3)}_i$ boxes and the third from last row containing $l^{(1)}_i + l^{(2)}_i+ l^{(3)}_i = \ell$ boxes. 

\begin{proof}
If $(l^{(1)}_{i}, l^{(2)}_{i}, l^{(3)}_{i})$ is a low-row of $\varrho(i)$, we analyze the possible low-row of $\varrho(i-1)$ which is obtained after filling $\varrho(i)$ with content $c_{i-1}$ of flavor $i-1$ in the reverse fill method. We must consider three cases of different amounts of content of flavor $i-1$ relative to the low-row of $\varrho(i)$. Recall, in the reverse fill method \ref{reverse}, we begin by using the largest content flavor and continue to fill by decreasing in flavor value (we start with $n$ and then place flavor of size $n-1$ and continue). For an example of each of the cases see Example \ref{lowrowexamples}.\\

\textbf{Case One:} 
$0< c_{i-1} \leq  l^{(3)}_{i}$.  After placing the $c_{i-1}$ content of flavor $i-1$, the low-row of $\varrho(i-1)$ will be $(l^{(1)}_{i-1}, l^{(2)}_{i-1}, l^{(3)}_{i-1}) = (l^{(1)}_{i}, l^{(2)}_{i}+c_{i-1}, l^{(3)}_{i}-c_{i-1})$.

\textbf{Case Two:} $ l^{(3)}_{i} \leq c_{i-1} \leq  l^{(2)}_{i} + l^{(3)}_{i}$.  After placing the $c_{i-1}$ content of flavor $i-1$, the low-row of $\varrho(i-1)$ will be $(l^{(1)}_{i-1}, l^{(2)}_{i-1}, l^{(3)}_{i-1}) = (l^{(1)}_{i}+c_{i-1}- l^{(3)}_{i}, l^{(2)}_{i}+2(l^{(1)}_{i})-c_{i-1}, 0)$. 

\textbf{Case Three:} $l^{(2)}_{i} + l^{(3)}_{i} \leq c_{i-1} \leq  l^{(1)}_{i} + l^{(2)}_{i} + l^{(3)}_{i} = \ell $. After placing the $c_{i-1}$ content of flavor $i-1$, the low-row of $\varrho(i-1)$ will be $(l^{(1)}_{i-1}, l^{(2)}_{i-1}, l^{(3)}_{i-1}) = (c_{i-1}-l^{(2)}_{i}-l^{(3)}_{i}, \ell-(c_{i-1}-l^{(2)}_{i}-l^{(3)}_{i})-(l^{(3)}_{i}), l^{(3)}_{i})$. \\

Analyzing the result of each case, we see that placing content $c_{i-1}$ will have one of the following effects, $0 \leq l^{(3)}_{i-1}  \leq l^{(3)}_{i}$ (case one and two) or $0 < l^{(1)}_{i-1} \leq l^{(1)}_i $ (case three). Since the initial low-row of $\varrho$ is  $(l^{(1)}, l^{(2)}, l^{(3)}) =(\ell-p, 0, p)$. We have $ l^{(3)}_i \leq l^{(3)}=p$ and $\leq l^{(1)}_i  \leq l^{(1)}\ell-p$ for any shape $\varrho(i)$. Hence, we can conclude the statement of the lemma.
\end{proof}

\begin{example}\label{lowrowexamples}
We begin with shape $\varrho(i) = (9,9,5,3)$ with low-row $(4,3,2)$ shaded in the image below. We demonstrate the effect on the low-row after filling with content $c_{i-1}$ with amount relative to each case above.
$$
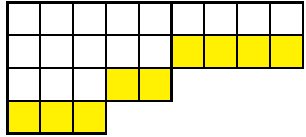
\ytableausetup
{mathmode, boxsize=1em}
  \begin{ytableau}
\text{} & \text{} & \text{} & \text{} & \text{} & \text{} & \text{} & \text{} & \text{} \\
\text{} & \text{} & \text{} & \text{} & \text{} & *(yellow)\text{} & *(yellow)\text{} & *(yellow)\text{} &*(yellow) \text{} \\
\text{} & \text{} & \text{} & *(yellow)\text{} & *(yellow)\text{} \\
 *(yellow)\text{} & *(yellow)\text{} & *(yellow) \text{} 
 \end{ytableau} 
$$

In the following cases, the boxes in the low-row of $\varrho(i)$ which are filled with content $i-1$ are darkly shaded and the new low-row for $\varrho(i-1)$ is lightly shaded.\\

\textbf{Case One:} $ c_{i-1} = 2 < 3$. After filling the low-row $(4,2,3)$ with $2$ content values the low-row becomes $(4,4,1)$

$$
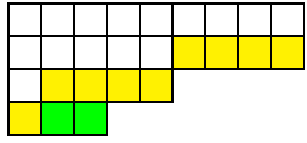
\ytableausetup
{mathmode, boxsize=1em}
  \begin{ytableau}
\text{} & \text{} & \text{} & \text{} & \text{} & \text{} & \text{} & \text{} & \text{} \\
\text{} & \text{} & \text{} & \text{} & \text{} & *(yellow)\text{} & *(yellow)\text{} & *(yellow)\text{} &*(yellow) \text{} \\
\text{} &*(yellow) \text{} & *(yellow)\text{} & *(yellow)\text{} & *(yellow)\text{} \\
 *(yellow)\text{} & *(green)\text{} & *(green) \text{} 
 \end{ytableau} 
$$

\textbf{Case Two:} $3 \leq 4=c_{i-1} \leq  5$. After filling the low-row $(4,2,3)$ with $4$ content values the low-row becomes $(5,4,0)$

$$
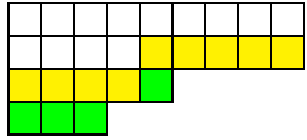
\ytableausetup
{mathmode, boxsize=1em}
  \begin{ytableau}
\text{} & \text{} & \text{} & \text{} & \text{} & \text{} & \text{} & \text{} & \text{} \\
\text{} & \text{} & \text{} & \text{} & *(yellow)\text{} & *(yellow)\text{} & *(yellow)\text{} & *(yellow)\text{} &*(yellow) \text{} \\
*(yellow)\text{} &*(yellow) \text{} & *(yellow)\text{} & *(yellow)\text{} & *(green)\text{} \\
 *(green)\text{} & *(green)\text{} & *(green) \text{} 
 \end{ytableau} 
$$

\textbf{Case Three:} $5 \leq 7=c_{i-1} \leq  9$. After filling the low-row $(4,2,3)$ with $7$ content values the low-row becomes $(2,4,3)$

$$
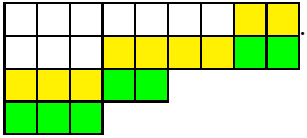
\ytableausetup
{mathmode, boxsize=1em}
  \begin{ytableau}
\text{} & \text{} & \text{} & \text{} & \text{} & \text{} & \text{} & *(yellow)\text{} & *(yellow)\text{} \\
\text{} & \text{} & \text{} & *(yellow)\text{} & *(yellow)\text{} & *(yellow)\text{} & *(yellow)\text{} & *(green)\text{} &*(green) \text{} \\
*(yellow)\text{} & *(yellow)\text{} & *(yellow)\text{} & *(green)\text{} & *(green)\text{} \\
 *(green)\text{} & *(green)\text{} & *(green) \text{} 
 \end{ytableau} .
$$
\end{example} 

\begin{proof}[Proof of Proposition \ref{combinedfillproperyoungtab}]
We demonstrate the construction of a proper Young tableau using the forward fill \ref{forward} and reverse \ref{reverse} fill methods. 

Begin by filling $\varrho$ with content $(c_{2k+1}, ..., c_n)$ in the reverse fill method. With the conditions on the sum $\sum_{i=1}^{n} c_i = 2k\ell + 2p$ and $c_i \leq \ell$ this placement will have strictly increasing columns and weakly increasing rows. Recall, that in the reverse fill method, each of our content flavors $(c_{2k+2}, ..., c_n)$ is placed in a low-row of size $\ell$, the boxes in a low-row are all in distinct columns which gives such a filling strictly increasing columns.

Now, place the remaining content $(c_1, ..., c_{2k+1})$ in the shape $\varrho(2k+2)$ in the forward fill method \ref{forward}. With the forward fill method, we place content $(c_1, ..., c_{2k+1})$ in the rows of $\varrho(2k+2)$ in increasing flavors, so rows will always be weakly increasing. Furthermore, since each $c_i \leq \ell$ the content flavors in any fixed column will necessarily increase across two rows if the higher row contains $\ell$ boxes. For if content of flavor $i$ appears in both boxes of a column of such rows, the content $c_i$ must have filled in all boxes in the first row to the right of this column and all boxes to the left of this column in the second row, such an arrangement requires more than $\ell$ content amounts.

We see from Lemma \ref{sizeoflowrow} that $\varrho(2k+2)$ has shape $$(\ell^{j}, \ell_{2k+2}^{(3)} + \ell_{2k+2}^{(2)}, \ell_{2k+2}^{(3)})$$ for some $j \leq 2k$ and low-row $(\ell_{2k+2}^{(1)}, \ell_{2k+2}^{(2)}, \ell_{2k+2}^{(3)})$. The only possible consecutive rows of $\varrho(2k+2)$ each with less than $\ell$ boxes are the lowest two rows (if $\ell_{2k+2}^{(3)}> 0$) of size $\ell_{2k+2}^{(3)} + \ell_{2k+2}^{(2)}$ and $\ell_{2k+2}^{(3)}$.

We now analyze the content in the lowest two rows of $\varrho(2k+1)$ when $\ell_{2k+2}^{(3)}>0$. We see that content within any column of such rows will necessarily increase unless, for some content of flavor $m$ from the content $(c_1, ..., c_{2k+1})$ we have that $c_m> l^{(2)}_{2k+2}+l^{(3)}_{2k+2}$ (i.e., the $c_m$ of flavor $m$ is contained in one column in the lowest two rows of $\varrho(2k+2)$).

Suppose flavor $m$ occurs in the lowest two rows of $\varrho(2k+2)$ in the same column and so $c_m > l^{(2)}_{2k+2}+l^{(3)}_{2k+2}$. We analyze possible flavors of this content $c_m$ which is assumed to over fill a row in the lowest two rows of $\varrho(2k+2)$.

First, suppose $m < 2k+1$. In this case, we must have $c_{m+1} < l^{(3)}_{2k+2}$ (since by assumption, the content $(c_1, ..., c_{2k+1})$ fills in $\varrho(2k+2)$). However, in order for $\varrho(2k+2)$ to have a low-row such that $l^{(3)}_{2k+2} >0$, we must have been in Case 3 of Lemma \ref{sizeoflowrow} for some content $c_{t-1}$ and $\varrho(t)$. This means $c_{t-1} > p \geq l^{(3)}_{2k+2}$ for some content flavor $c_{t-1}$ in $(c_{2k+2}, ...., c_n)$. These inequalities would give $c_{t-1} > l^{(3)}_{2k+2} > c_{m-1}$. This contradicts the weakly decreasing condition on the content $(c_1, ..., c_n)$. It must be that for such content $c_m$ having this assumed property, that $m=2k+1$. We are able to reduce the situation and only consider when $l^{(3)}_{2k+2} > 0$ and $c_{2k+1} > l^{(2)}_{2k+2}+l^{(3)}_{2k+2}$ (i.e., if $m=2k+1$).

With such a situation, we reason by induction on the size of the the first low-row of $\varrho(2k+2)$, denoted $l^{(1)}_{2k+2}$, and produce a tableau on $\varrho$ by now placing content $(c_{2k+1}, ..., c_n)$ in the reverse fill method. 

For the base case, consider when the first row of the low-row of $\varrho(2k+2)$ contains one box, $l^{(1)}_{2k+2}=1$. In this situation, we have $l^{(2)}_{2k+2}+l^{(3)}_{2k+2} = \ell -1$ (since the sum of all boxes in the low-row must add to $\ell$, see Def. \ref{lowrow}).

Using $\sum_{i=1}^{2k+1} c_i \leq 2k\ell +p$ (the assumption of this proposition), if $c_{2k+1} > l^{(2)}_{2k+2}+l^{(3)}_{2k+2} = \ell-1$, we would have $c_{2k+1} = \ell$ and so $\sum_{i=1}^{2k+1} c_i = 2k\ell + \ell$. This forces our content to be of the form $(\ell^{2k+1})$ which is already considered in Lemma \ref{sumequalp} producing a tableau.

For the inductive hypothesis, suppose that if $l^{(3)}_{2k+2} > 0$ and $l^{(1)}_{2k+2} \leq j$ for some integer $j$, then $c_{2k+1} \leq  l^{(2)}_{2k+2}+l^{(3)}_{2k+2}$ and a Young tableau is constructed in this method.

Now assume $l^{(3)}_{2k+2} > 0$, $l^{(1)}_{2k+2} = j+1$, and $c_{2k+1} > l^{(2)}_{2k+2}+l^{(3)}_{2k+2}$. Consider the content $\tilde{C} = (\ell, \ell, c_1, ..., c_{2k-1}, c_{2k}, c_{2k+1}, c_{2k+2}, ..., c_n)$ (give the first content flavor $-1$ and $0$) and consider $\rk(\mathbb{V}(\mathfrak{sl}_2, \tilde{C}, \ell))$. The sum of this content is, $\ell + \ell + \sum_{i=1}^n c_i = 2(k+1)\ell + 2p$. Denote, the sum and shape from Theorem \ref{mainIntro} and Lemma \ref{numberyoungtab} associated to $\mathbb{V}(\mathfrak{sl}_2, \tilde{C}, \ell)$ as $\tilde{\Lambda}$ and $\tilde{\varrho}$. We have $\tilde{\Lambda} = \sum_{i=2k}^n c_i = c_{2k} + c_{2k+1} + \Lambda$ and $\tilde{\varrho} = (\ell^{2k+2}, p, p).$ The shape $\tilde{\varrho}$ is the original diagram $\varrho$ with two rows of length $\ell$ vertically concatenated on top. And further $\tilde{\varrho}(2k+2) $ is the shape $\varrho(2k+2)$ with two extra rows of length $\ell$. 

We then continue with the reverse fill method with $c_{2k}$ and $c_{2k+1}$ as initially prescribed (using all content from the sum $\tilde{\Lambda}$). Recall, we have assumed that $c_{2k+1} > l^{(2)}_{2k+2}+l^{(3)}_{2k+2}$ and so we are in case three of Lemma \ref{sizeoflowrow}. The low-row of $\tilde{\varrho}(2k+1),$ after placing the content of flavor $2k+1$, is necessarily such that $l^{(1)}_{2k+1} < l^{(1)}_{2k+2} = j+1$. Furthermore, after placing $c_{2k}$ into $\tilde{\varrho}(2k+1)$ we have that $l^{(1)}_{2k} \leq j$. By our inductive hypothesis, it must be that $c_{2k} \leq l^{(2)}_{2k}+l^{(3)}_{2k}$ and so a Young tableau can be produced with $(\ell, \ell, c_1, ..., c_n)$ on $\tilde{\varrho}$. The first and second rows must be filled with only flavors $-1$ and $0$ respectively. By removing these two rows we have constructed a Young tableau on $\varrho$ with content $(c_1, ..., c_n)$.

To show that this Young tableau is \textit{proper}, we use Lemma 1 from \cite{Bertram}. This lemma is equivalent to the statement that the Young tableau constructed in the above method is proper if and only if the low-row of $\varrho(i)$ has at most three non-empty rows (for any value of $i$). We showed in Lemma \ref{sizeoflowrow} that this is indeed the case. 
\end{proof}

\begin{lemma}\label{cfmmodified}
With notation from Lemma \ref{numberyoungtab}, a Young tableau produced on $\varrho$ by placing content $(c_{2k+2}, ..., c_n)$ in a reverse fill method (as in the proof of \ref{combinedfillproperyoungtab}) can be modified when $\Lambda > p$ and content is not maximal to produce a new Young tableau with the same content.
\end{lemma}

\begin{proof}  
Consider the proper Young tableau produced in Proposition \ref{combinedfillproperyoungtab} by placing content $(c_{2k+2}, ..., c_n)$ in shape $\varrho$ in the reverse fill method. Denote by $r$ the lowest row of this tableau with content of flavors $2k+1$ and smaller (i.e., the lowest row of the Young tableau without values from content $(c_{2k+2}, ..., c_n)$). Since $\Lambda > p$ and content is not maximal, $r > 2k+1$ and this row must contain more than one type of flavor. There must be a column between rows $r$ and $r+1$ in which content flavors increase by an amount larger than one, denote this column $l$. Denote the flavors in column $l$ and rows $r$ and $r+1$ as $a$ and $b$, so $a+1 < b$. 

Consider the largest content flavor in row $r$ with flavor smaller than $b$, denote this flavor $y$. Note that $y \geq a$ and in the case $a = y$, then row $r+1$ must contain a flavor strictly smaller than $b$, particularly $a+1$. For clarity in the following argument we assume $y> a$, a similar argument holds if $y=a$. Select the largest box, $B_{(r, l')}$, in row $r$ containing flavor $y$; note that $B_{(r, l')}$ will be such that $l < l'$ and the flavor in $B_{(r+1, l')}$, directly below $B_{(r, l')}$, will either contain content of flavor larger than $b$ (as $c_b < \ell$, content is not maximal) or such will not be a box in $\varrho$.  

Now select the smallest box in row $r+1$ with flavor $y+1 \leq b$, denote this flavor $z$ and the selected box $B_{(r+1, l'')}$. Since $a$ and $b$ are such that $1 < b-a$, we know that two such flavors $y$ and $z$ exist between $a$ and $b$ (with the possibility that $a=y$ or $b=z$ but not both) and that $l'' \leq l <  l'$. 

To modify the given tableau, switch the flavors of $y$ and $z$ in the selected boxes. Since $y, z \leq b$ and $B_{(r+1, l')}$ contains content flavor strictly larger than $b$, we know that placing content of flavor $z \leq b$ in $B_{(r, l')}$ will maintain the strictly increasing behavior of column $l'$. Since $a < y,z $ and all content flavors in row $r$ to the left of box $B_{(r, l)}$ will have content values $a$ or smaller the modification of content in row $r+1$ will still maintain the strictly increasing behavior in column $l''$. 
\end{proof}

\begin{example}
Consider the vector bundle $V= \mathbb{V}(\mathfrak{sl}_2, (10, 8, 8, 7, 6, 3, 1, 1), 10).$
In this case, we have $n=8$, $\ell = 10$, $k=2$, $p = 2$ and $\Lambda= \sum_{i=6}^8 c_i = 5 > p$. Using the method of producing a tableau as described in the proof of Proposition \ref{combinedfillproperyoungtab} with our given data, the Young tableau is created. The boxes with content from $\Lambda$, $(c_6, c_7, c_8) = (3, 1, 1)$, are lightly shaded. Content to be switched specified by the description in Lemma \ref{cfmmodified} are darkly shaded. As one can check, switching the flavors in these boxes does produce a new proper Young tableau. \\

$$
\ytableausetup
{mathmode, boxsize=1em}
  \begin{ytableau}
\text{1} &  \text{1} & \text{1} & \text{1} & \text{1} & \text{1} & \text{1} & \text{1} & \text{1} & \text{1} \\
\text{2} & \text{2} & \text{2} & \text{2} & \text{2} & \text{2} & \text{2} & \text{2} & \text{3} & \text{3} \\
\text{3} & \text{3} & \text{3} & \text{3} & \text{3} & \text{3} & \text{4} & \text{4} & \text{4} & *(green)\text{4} \\
\text{4} &  \text{4} & \text{4} &  *(green)\text{5} & \text{5} & \text{5} &  \text{5} &  \text{5} &  \text{5} &*(yellow)  \text{6} \\
 *(yellow)\text{6} & *(yellow)\text{6} \\
  *(yellow)\text{7} &*(yellow)\text{8} 
\end{ytableau} 
$$

\end{example}

We are now able to conclude the following result.
\begin{proposition}\label{gp}
If $\Lambda > p$ and content $\mu$ is not maximal then $\rk(\mathbb{V}_1) > 1$.  
\end{proposition}

\begin{proof}
With the assumptions on $\Lambda$ and $\mu$, Lemma \ref{combinedfillproperyoungtab} and Lemma \ref{cfmmodified} show that more than one proper Young tableau can be produced on $\varrho$ with content $\mu$. By Lemma \ref{numberyoungtab}, we can conclude $\rk(\mathbb{V}_1) > 1$.
\end{proof}

\subsection{ Main Theorem}

We are now able to show the proof of the main theorem.

\begin{proof}[Proof of Theorem \ref{mainIntro}]
By collecting the results from Lemma \ref{sumlessp}, Lemma \ref{sumequalp}, Lemma \ref{case2leftarrow}, and Proposition \ref{gp} we can conclude the statement of the main theorem for the case $m=1$. By the scaling property in Proposition \ref{bgkscaling} the full statement of Theorem \ref{mainIntro} is shown.
\end{proof}

\section{Decomposition of First Chern Classes of Rank One Bundles}


We now prove our final result on the decomposition of rank one bundles of type $\mathfrak{sl}_{2m}$ with rectangular weights into effective sums of level one divisors. We use the notation from Theorem \ref{mainIntro} and the vector bundles in Definition \ref{vweightvector}.

\begin{theorem} \label{divisors}
In the case $\rk(\mathbb{V}_m)=1$, we have the following decomposition as a sum of level one divisors
\begin{enumerate}
\item[1. ] \label{divisor1} Content is not maximal:
$$c_1(\mathbb{V}_m)= \sum_{i=1}^{2k+1}(\ell-c_i) \cdot \text{c}_1(V_{i, 0}) + \sum_{j=2k+2}^{n}c_j \cdot \text{ c}_1(V_{0, j}).$$
\item[2. ] Content is maximal:

\begin{enumerate}
\item If $n$ is odd,
$$c_1(\mathbb{V}_m) = (\ell-p)\cdot c_1(V_{n-2, 0}) + (p-c_{n})\cdot c_1(V_{0, n-1}) + (p-c_{n-1})\cdot c_1(V_{0,n}) +(p-c_{n-2}) \cdot c_1(V_{n-2,\{n-1,n\} }).$$ 
\item If $n$ is even,
$$c_1(\mathbb{V}_m) = (p) \cdot c_1(V_{0, \{n-1, n\}})  + (c_{n-2}-p) \cdot c_1(V_{n-1,0}) + (c_{n-1}-p)\cdot c_1(V_{n-2, 0}) + (c_{n}-p) \cdot c_1(V_{\{n-2, n-1\},n }). $$

\end{enumerate} 

\end{enumerate}
\end{theorem}
First, we show Lemma \ref{removecolumn} to describe how to decompose the weights of $c_1(\mathbb{V}_m)$. We then describe the content in the columns of the unique tableau from Lemma \ref{sumequalp} (the case with maximal content follows similarly). Using these results and Proposition 1.2 from \cite{bgm}, we conclude the theorem statement.

\subsection{Lemma for decomposing divisors of rank one $\mathfrak{sl}_{2m}$ bundles with rectangular weights}
As is standard, we denote the transpose of a diagram $\varrho$ as $\tilde{\varrho}= (\tilde{\varrho}^{(1)}, ..., \tilde{\varrho}^{(\ell)})$. With this notation, $\tilde{\varrho}^{(i)}$ is the number of boxes in the $i^{th}$ column of $\varrho$. 

\begin{lemma}\label{removecolumn}
If $T$ is the only possible Young tableau with content $\mu= (c_1, ..., c_n)$ on shape $\tilde{\varrho}=((2k+2)^q, 2k^{m})$ for some integers $k, q, m \geq 0$, then the tableau obtained from simply removing the final column of $T$, denote this tableau $T^{-c}$,  is the only tableau with shape $\tilde{\varrho}^{-c}= ((2k+2)^{q}, 2k^{m-1})$ and content $\mu^{-c}=(c'_1, ..., c'_n)$ obtained from content $\mu$ by removing the content flavors from the final column of $T$. 
\end{lemma}

\begin{proof}
Suppose $T$ is the only Young tableau one shape $\tilde{\varrho}$ and content $\mu$ and $T^{-c}$ is the tableau described in the lemma statement. Suppose $T'$ is a Young tableau on shape $\tilde{\varrho}^{-c}$ with content $\mu^{-c}$ different than the tableau $T^{-c}$.

Since $T$ was the only possible Young tableau on $\varrho$ with content $\mu$, if we were to concatenate the removed column of $T$ to the largest column of $T'$ we should not have a Young tableau (as we have assumed $T$ is the only such and $T' \neq T^{-c}$). There must be some content flavor $i$ in the final column of $T'$ which is larger than the content to the right of this row in the removed column (making the concatenated tableau not have weakly increasing property in rows). 

We can thus assume that $T'$ and $T^{-c}$ differ by a flavor in their final $m-1$ column. Denote the box of $\varrho^{-c}$ containing this differing content $B_{x,m-1}$. We can assume that the content flavor in $B_{x,m}$ of $T'$ is $i$ and in $B_{x,m-1}$ of $T^{-c}$ is $i-1$. Now, both $T'$ and $T^{-c}$ have the same content and number of boxes and particularly, the number of boxes larger than $B_{x,m-1}$ is the same. Furthermore, the boxes of $T'$ and $T^{-c}$ larger than $B_{x,m-1}$ must contain content of flavor $i$ or larger and content of flavor $i-1$ or larger respectively. Similarly, the boxes of $T'$ and $T^{-c}$ smaller than $B_{x,m-1}$ must contain content of flavor $i$ or smaller and content of flavor $i-1$ or smaller respectively. The size of content with flavor $i$ or larger in $T'$ is thus more than the amount of content flavors $i -1$ or larger of $T^{-c}$ (including the content of flavor $i$ in $B_{x,m-1}$ of $T'$). This contradicts the two tableau $T'$ and $T^{-c}$ having the same content. It must be the case that $T^{-c}$ is the only Young tableau with the given conditions. 

\end{proof}

\subsection{Description of column weight content when $\Lambda =p$}\label{columncontent}

Consider the $j^{th}$ column of the Young tableau constructed in Lemma \ref{sumequalp}. 

If $p < j \leq \ell$, then this column contains $2k$ boxes with content from flavors $1$ to $2k+1$ which are strictly decreasing. There is at most one occurrence between consecutive rows in this column where the content flavors increase by two and all other content flavors increase by exactly one. Denote $i_j$ the flavor missing by this increase of size two from column $j$. The transpose of column $j$ has the following form. 
\tiny
$$
\ytableausetup
{mathmode, boxsize=3em}
  \begin{ytableau}
\text{1} & \text{2} & \none[\dots] & i_j-1 & i_j+1 & \none[\dots] & 2k
\end{ytableau}.
$$
\normalsize
\noindent The weight vector given by these flavors is $\vec{v}_{i_j, 0}$ (see Def. \ref{vweightvector}).
.

Now suppose $1 \leq j \leq p$, then the $j^{th}$ column contains $2k+2$ boxes such that the content in the first $2k+1$ boxes decrease by one in value across each row and the final box contains content of flavor from $\{2k+2, ..., n\}$. Let $i_j$ denote this flavor. The weight vector given by the flavors in this column is given by $\vec{v}_{0, i_j}$ (see Def. \ref{vweightvector}).  


%
%

\subsection{Proof of divisor decomposition}
We now prove our result on the decomposition.

\begin{proof}[Proof of Theorem \ref{divisors}]
Let $\mathbb{V}_1$ be a rank one bundle with $\Lambda =p$. As in the statement of Lemma \ref{removecolumn}, let $\vec{\lambda}^{-c}$ be the weight vector obtained by removing the content from $\vec{\lambda}$ in the final column of the tableau produced in Lemma \ref{sumequalp}. Lemma \ref{removecolumn} and Lemma \ref{numberyoungtab} implies $$\rk(\mathbb{V}(\mathfrak{sl}_2, \vec{\lambda}^{-c}, \ell-1)) = 1.$$ Furthermore, since the following is at level one (recall \cite[5.2.1]{Fakh}), we have $$\rk(\mathbb{V}(\mathfrak{sl_2}, \vec{v}_{i_{\ell},0}, 1)) =1,$$ where $\vec{v}_{i_{\ell},0}$ is the weight vector associated to column $\ell$. 

By Proposition 1.2 in \cite{bgm}, we have:
$$c_1(\mathbb{V}_1) = c_1(\mathbb{V}(\mathfrak{sl}_2, \vec{\lambda}^{-c}, \ell-1) + c_1(V_{i_{\ell}, 0}).$$   

We carry out this same process now with the rank one conformal blocks vector bundle $\mathbb{V}(\mathfrak{sl}_2, \vec{\lambda}^{-c}, \ell-1)$. This process continues until we have completely decomposed the original divisors into a sum of level one divisors. Observing the behavior of the weight content in column $j$ with $j \leq p$ and $ p < j$ of the tableau produced in Lemma \ref{sumequalp} allows one to conclude the description of (1) in Theorem \ref{divisor1}.

\end{proof}

\begin{example}
We demonstrate the decomposition of $c_1(V)$ as in Theorem \ref{divisors} with $V = \mathbb{V}(\mathfrak{sl}_2, (9,8,8,8,8,8,8, 2,1), 9)$. We have $|\vec{\lambda}| = 2((3)9+3)$, $k=3$, and $p=3$. As $\sum_{i=2(3)+1}^{10} c_i = 2 +1 = 3 = p$ Theorem \ref{mainIntro} implies that indeed $\rk(V) = 1$.

The unique Young tableau from Lemma \ref{sumequalp} is the following: 
$$
\ytableausetup
{mathmode, boxsize=1em}
  \begin{ytableau}
\text{1} & \text{1} & \text{1} & \text{1} & \text{1} & \text{1} & \text{1} & \text{1} & \text{1} \\
\text{2} & \text{2} & \text{2} & \text{2} & \text{2} & \text{2} & \text{2} & \text{2} & \text{3} \\
\text{3} & \text{3} & \text{3} & \text{3} & \text{3} & \text{3} & \text{3} & \text{4} & \text{4} \\
\text{4} & \text{4} & \text{4} & \text{5} & \text{4} & \text{4} & \text{5} & \text{5} & \text{5} \\
\text{5} & \text{5} & \text{5} & \text{5} & \text{5} & \text{6} & \text{6} & \text{6} & \text{6} \\
\text{6} & \text{6} & \text{6} & \text{6} & \text{7} & \text{7} & \text{7} & \text{7} & \text{7} \\
 \text{7} & \text{7} & \text{7} \\
 \text{8} & \text{8} & \text{9} 
\end{ytableau} 
$$

Each column of this Young tableau gives the weight content associated to the Young tableau for the divisors of level one in the decomposition as in Theorem \ref{divisors}. The content $i_j$ (or those content $i_j-1$ and $i_j+1$) associated to column $j$ is highlighted in each column (Section \ref{columncontent}).

$$
\ytableausetup
{mathmode, boxsize=1em}
  \begin{ytableau}
\text{1} & \none & \text{1} & \none & \text{1} & \none & \text{1} &\none &  \text{1} & \none & \text{1} & \none & \text{1} &\none &  \text{1} &\none & *(yellow) \text{1} \\
\text{2} & \none & \text{2} & \none & \text{2} & \none & \text{2} & \none & \text{2} & \none & \text{2} & \none & \text{2} & \none & *(yellow)\text{2} & \none & *(yellow)\text{3} \\
\text{3} & \none & \text{3} & \none & \text{3} & \none & \text{3} & \none & \text{3} &\none &  \text{3} &\none &  *(yellow) \text{3} & \none & *(yellow)\text{4} & \none & \text{4} \\
\text{4} & \none & \text{4} &\none &  \text{4} & \none & \text{5} & \none & \text{4} & \none & *(yellow) \text{4} & \none & *(yellow)\text{5} &\none &  \text{5} & \none & \text{5} \\
\text{5} & \none & \text{5} &\none &  \text{5} &\none &  \text{5} &\none &  *(yellow) \text{5} & \none & *(yellow)\text{6} & \none & \text{6} & \none & \text{6} & \none & \text{6} \\
\text{6} &\none &  \text{6} & \none & \text{6} & \none & *(yellow) \text{6} & \none &  *(yellow) \text{7} & \none &  \text{7} & \none &  \text{7} & \none & \text{7} & \none & \text{7} \\
 \text{7} & \none & \text{7} & \none & \text{7} \\
 *(yellow)\text{8} &\none & *(yellow) \text{8} & \none &  *(yellow)\text{9} 
\end{ytableau} 
$$

From this, we see that the divisor decomposes as $$c_1(V) = c_1(V_{2,0}) + c_1(V_{3,0}) + c_1(V_{4,0}) + c_1(V_{5,0})  + c_1(V_{6,0})  + c_1(V_{7,0}) + 2(c_1(V_{0, 8})) + c_1(V_{0, 9}) .$$ 
\end{example}

  \section{Rank one vertical scaling}\label{scalingproof}
  If $\mathbb{V}=\mathbb{V}(\sL_{r+1},\vec{\lambda},\ell)$ is a bundle of rank one on $\overline{\operatorname{M}}_{0,n}$, then by a quantum generalization  of a conjecture of Fulton (see \cite{BGMB}),  bundles $\Bbb{V}[m]=\mathbb{V}(\sL_{r+1},m\vec{\lambda},m\ell)$ obtained by scaling Lie data horizontally (see Def~\ref{Stetching}), are also rank one for all $m$. Moreover, the first Chern classes are related by the identity 
  $c_1(\mathbb{V}[m])=m \ c_1(\mathbb{V})$ \cite{GiansiracusaGibney}.

 Belkale, Gibney and Kazanova  have shown that analagous rank and Chern class scaling symmetries hold when the data defining $\mathbb{V}$ is scaled vertically.   The precise statement and proof, communicated by them, are given in Proposition \ref{GITScaling}.
 
 \begin{definition}\label{Stetching}
For $\lambda_i=\sum_{j=1}^rc_j \omega_j  \in P_{\ell}(\sL_{r+1})$, and $m\in \mathbb{N}$, set $m\lambda_i=\sum_{j=1}^r (m  c_j ) \omega_j \in P_{m \ell}(\sL_{r+1})$.
 Given $\mathbb{V}=\mathbb{V}(\sL_{r+1},\vec{\lambda},\ell)$, set
$\Bbb{V}[m]=\mathbb{V}(\sL_{r+1},m\vec{\lambda}, m \ell)$, where $m\vec{\lambda}=(m\lambda_1,\ldots,m\lambda_n)$.
\end{definition}

 \begin{proposition}\label{GITScaling}[BGK] If $\operatorname{rank}\mathbb{V} =1$, then $\operatorname{rk}\mathbb{V}_m =1$, and
 $c_1(\mathbb{V}_m)=m \ c_1(\mathbb{V})$
\end{proposition}

\begin{proof}
The second part of Prop \ref{GITScaling}, will follow from Lemma \ref{4PointGITScaling2}, which will imply that  the divisors $m \ c_1\mathbb{V}$ and $c_1\mathbb{V}_m$ intersect every F-Curve in the same degree.

If $\mathbb{V}$  has rank one, then by taking its strange dual partner, one gets the rank one bundle
$\mathbb{V}(\sL_{\ell}, (\lambda^T_{1},\ldots,\lambda^T_{n-1}, \widetilde{\lambda^{T}}_n), r+1)$, where  $ \widetilde{\lambda^{T}}_n$ accounts a necessary cyclic twisting.   Now by Fulton's Conjecture, $\mathbb{V}(\sL_{\ell}, m\lambda^T_{1},\ldots,m\lambda^T_{n-1}), m\widetilde{\lambda^{T}}_n,m(r+1))$
also has rank one.  By again applying strange duality to this bundle, one obtains 
$\mathbb{V}(\sL_{m(r+1)}, (\lambda_{m(1)},\ldots, \lambda_{m(n)}), \ell)$ which therefore has rank one. 

\end{proof}

\begin{lemma}\label{4PointGITScaling2}[BGK] If $\operatorname{rank}\mathbb{V}(\sL_{m(r+1)},(\lambda_1(m),\ldots,\lambda_4(m)),\ell) =1$, then
$$m \ \operatorname{deg}\mathbb{V}(\sL_{r+1},(\lambda_1,\ldots, \lambda_4),\ell)
= \operatorname{deg}\mathbb{V}(\sL_{m(r+1)},(\lambda_1(m),\ldots,\lambda_4(m)),\ell).$$

\end{lemma}

\setlength\abovedisplayskip{5pt plus 2pt minus 2pt}
\begin{proof}  Assuming that $\operatorname{rk}\mathbb{V}(\sL_{m(r+1)},(\lambda_1(m),\ldots,\lambda_4(m)),\ell)=1$ for all $m \ge 1$,  we note that
the restriction data which determines $$\operatorname{deg}\mathbb{V}(\sL_{r+1},(\lambda_1,\allowbreak \ldots, \lambda_4),\ell),$$
will scale to give the restriction data that defines $$\operatorname{deg}\mathbb{V}(\sL_{m(r+1)},(\lambda_1(m),\ldots,\lambda_4(m)),\ell).$$

That is, for $\lambda$ in the second part of Fakhruddin's formula
\begin{multline}
\operatorname{deg}\mathbb{V}(\sL_{r+1},(\lambda_1,\ldots,\lambda_4),\ell)\\
= \frac{1}{2(r+1+\ell)}\Big(\sum_{i=1}^4c(\lambda_i)-\sum_{\lambda \in P_{\ell}(\sL_{r+1})}c(\lambda) \prod_{\{abcd\}=\{1234\}}\operatorname{rk}
\mathbb{V}(\sL_{r+1},(\lambda_a,\lambda_b,\lambda),\ell)\operatorname{rk}\mathbb{V}(\sL_{r+1},(\lambda_c,\lambda_d,\lambda^{\star}),\ell)\Big),
\end{multline}

 one just takes $\lambda(m)$ for the corresponding summand in  $\operatorname{deg}\mathbb{V}(\sL_{m(r+1)},(\lambda_1(m),\ldots,\lambda_4(m)),\ell)$. Similarly, with horizontal scaling, the restriction data which determines $\operatorname{deg}\mathbb{V}(\sL_{m(r+1)},(\lambda_1(m),\allowbreak \ldots, \lambda_4(m)),\ell)$,
will scale to give the restriction data that defines $\operatorname{deg}\mathbb{V}(\sL_{m(r+1)},(m\lambda_1(m),\allowbreak \ldots, m\lambda_4(m)),m\ell)$,.

The standard formula for the Casimir numbers for $\mathfrak{sl}_{r+1}$ is 
 \begin{equation}\label{mir}
 c_{r+1}(\lambda)=\frac{1}{r+1}\sum_{i=1}^r(r+1-i)ic_i^2+\frac{2}{r+1}\sum_{1 \le i < j \le r}(r+1-j)i c_i c_j + \sum_{i=1}^r(r+1-i)i c_i.
 \end{equation}
 
We observe the behavior of the Casimir operator for $\mathfrak{sl}_{r+1}$ and weight $\lambda$ and $\mathfrak{sl}_{m(r+1)}$ and scaled weight $m\lambda(m)$:
  $$c_{m(r+1)} (m\lambda(m)) = m^3c_{r+1} (\lambda).$$
 
Now, we substitute \ref{mir} and use Fulton's conjecture for degrees. We then further substitute the above relationship of the casimir operator and a use Fulton's conjecture for rank one. 

\begin{multline}
 \operatorname{deg}\mathbb{V}(\sL_{m(r+1)},(\lambda_1(m),\ldots,\lambda_4(m)),\ell)\\=\frac{1}{m}\operatorname{deg}\mathbb{V}(\sL_{m(r+1)},(m\lambda_1(m),\ldots,m\lambda_4(m)),m\ell)\\
=\frac{1}{m} \frac{1}{(m(r+1)+m\ell)} \Big( \sum_{j=1}^4 c_{m(r+1)}(m\lambda_j(m))\\
-\sum_{\lambda=\sum_{i=1}^rc_i \omega_i }c_{m(r+1)}(m\lambda(m))\  \prod_{\{abcd\}=\{1234\}} \operatorname{rk}\mathbb{V}_{m\lambda_a(m),m\lambda_b(m),m\lambda(m)}\operatorname{rk}\mathbb{V}_{m\lambda_c(m),m\lambda_d(m),m\lambda^*(m)} \Big)\\
=\frac{1}{m^2} \frac{1}{((r+1)+\ell)} \Big( \sum_{j=1}^4 m^3c_{r+1}(\lambda_j)\\
-\sum_{\lambda=\sum_{i=1}^rc_i \omega_i }m^3c_{r+1}(\lambda)\  \prod_{\{abcd\}=\{1234\}} \operatorname{rk}\mathbb{V}_{\lambda_a(m),\lambda_b(m),\lambda(m)}\operatorname{rk}\mathbb{V}_{\lambda_c(m),\lambda_d(m),\lambda^*(m)} \Big)\\
\end{multline}

Now by canceling out all redundant factors of $m$, we are left with $$m \operatorname{deg}\mathbb{V}(\sL_{r+1},(\lambda_1,\ldots,\lambda_4),\ell).$$
\end{proof}

\begin{corollary}\label{converse}
For a $\mathfrak{sl}_2$ bundle $\mathbb{V}=\mathbb{V}(\mathfrak{sl}_{2},(c_1\omega_1, ..., c_n\omega_1), \ell)$, we have $\rk\mathbb{V}=1 $ if and only if $\rk{\mathbb{V}_m} = 1$. 
\end{corollary}

\begin{proof}
The first implication is Proposition \ref{bgkscaling}.

 For the converse let $\mathbb{V}=\mathbb{V}(\mathfrak{sl}_{2m},(c_1\omega_m, ..., c_n\omega_m), \ell)$ and $\mathbb{V}_k=\mathbb{V}(\mathfrak{sl}_{2mk},(c_1\omega_{mk}, ..., c_n\omega_{mk}), \ell)$. Using \ref{bgkscaling}, if $\rk(\mathbb{V}) = 1$ then $\rk(\mathbb{V}_{1/m}) =1$. The proof of Proposition \ref{bgkscaling} holds with a non-integer scale when $\mathbb{V}$ is of the above form and so $\mathbb{V}_{1/m}$ is defined. 
 This shows the converse of Proposition \ref{bgkscaling} in the special case of interest.

\end{proof}

\end{document}